\documentclass[a4paper, 11pt]{article}

\usepackage[utf8]{inputenc}
\usepackage[T1]{fontenc}
\usepackage[english]{babel}

\usepackage{times, amssymb, amsthm, mathtools}
\usepackage{algorithm, algpseudocode}
\usepackage[hidelinks]{hyperref}
\usepackage{enumitem}
\usepackage{multirow}
\usepackage{siunitx} 

\usepackage[authoryear, sort, comma, longnamesfirst]{natbib}
\numberwithin{equation}{section} 
\counterwithin{figure}{section}
\counterwithin{table}{section}

\newtheorem{theorem}{Theorem}[section]
\newtheorem{lemma}[theorem]{Lemma}
\newtheorem{remark}{Remark}[section]
\theoremstyle{definition}
\newtheorem{definition}{Definition}[section]

\newtheorem{problem}{Problem}[section]

\algrenewcommand{\alglinenumber}[1]{\footnotesize Step #1:}

\newcommand{\ind}[1]{\mathrlap{\qquad #1}} 
\newcommand{\texteq}{\mathrm}
\newcommand{\card}[1]{\lvert #1 \rvert} 
\let\proglang=\textsf 

\def\N{{\mathbb N}}
\def\R{{\mathbb R}}

\def\A{\mathcal A}
\def\B{\mathcal B}

\def\D{\mathcal D}
\def\L{\mathcal L}
\def\U{\mathcal U}
\def\H{\mathcal H}

\def\x{\mathbf x}
\def\n{\mathbf n}
\def\zero{\mathbf 0}
\def\z{\mathbf z}

\def\Lt{\widetilde{\mathcal L}}
\def\Ut{\widetilde{\mathcal U}}

\def\Vt{\widetilde V}

\title{\bf ANOTHER SOLUTION FOR SOME OPTIMUM ALLOCATION PROBLEM}
\author{WOJCIECH WÓJCIAK\thanks{Warsaw University of Technology, Warsaw, Poland 
		\href{mailto:wojciech.wojciak.dokt@pw.edu.pl}{wojciech.wojciak.dokt@pw.edu.pl}}}

\begin{document}
	\maketitle

\begin{abstract}
	We derive optimality conditions for the optimum sample allocation problem in stratified sampling, formulated as the 
	determination of the fixed strata sample sizes that minimize the total cost of the survey, under the assumed level of 
	variance of the stratified $\pi$ estimator of the population total (or mean) and one-sided upper bounds imposed on sample 
	sizes in strata. In this context, we presume that the variance function is of some generic form that, in particular, covers the 
	case of the simple random sampling without replacement design in strata. The optimality conditions mentioned above will 
	be derived from the Karush-Kuhn-Tucker conditions. Based on the established optimality conditions, we provide a formal 
	proof of the optimality of the existing procedure, termed here as {\em LRNA}, which solves the allocation problem 
	considered. We formulate the {\em LRNA} in such a way that it also provides the solution to the classical optimum allocation 
	problem (i.e. minimization of the estimator's variance under a fixed total cost) under one-sided lower bounds imposed on 
	sample sizes in strata. In this context, the {\em LRNA} can be considered as a counterparty to the popular recursive 
	Neyman allocation procedure that is used to solve the classical problem of an optimum sample allocation with added 
	one-sided upper bounds. Ready-to-use \proglang{R}-implementation of the {\em LRNA} is available through our {\tt 
	stratallo} package, which is published on the Comprehensive R Archive Network (CRAN) package repository.
	
	\smallskip \noindent \textbf{Key words:} stratified sampling, optimum allocation, minimum cost allocation under upper 
	bounds, optimum allocation constant variance, optimum allocation under lower bounds, recursive Neyman algorithm.
\end{abstract}
	
\section{Introduction}
\label{sec:intro}

Let us consider a finite population $U$ consisting of $N$ elements. Let the parameter of principal interest of a single study 
variable $y$ in $U$ be denoted by $\theta$. This parameter is the population total (i.e. $\theta = \sum_{k \in U}\, y_k$, 
where $y_k$ denotes the value of $y$ for population element $k \in U$), or the population mean (i.e. $\theta = \tfrac{1}{N} 
\sum_{k \in U}\, y_k$). To estimate $\theta$, we consider the {\em stratified $\pi$ estimator}, i.e. the {\em $\pi$ estimator} of 
Horvitz and Thompson \citep[see, e.g.][Section 2.8, p.~42]{Sarndal} in {\em stratified sampling}. Under this well-known 
sampling technique, population $U$ is stratified, i.e. $U = \bigcup_{h \in \H}\, U_h$, where $U_h,\, h \in \H$, called strata, are 
pairwise disjoint and non-empty, and $\H = \{1, \ldots, H\}$ denotes a finite set of strata indices of size $H \geq 1$. The size of 
stratum $U_h$ is denoted $N_h,\, h \in \H$ and clearly $\sum_{h \in \H} N_h = N$. Probability samples of size $n_h \leq 
N_h,\, h \in \H$ are selected independently from each stratum according to chosen sampling designs, which are often the 
same in all strata. The resulting total sample is of size $n = \sum_{h \in \H}\, n_h \leq N$. It is well know that the {\em stratified 
$\pi$ estimator} $\hat{\theta}$ of $\theta$ and its variance $V_{\hat{\theta}}$ are expressed in terms of the first and second 
order inclusion probabilities (see, e.g. \citet[Result 3.7.1, p.~102]{Sarndal} for the case when $\theta$ is the population total). 
In particular, for several important sampling designs
\begin{equation}
	\label{eq:var}
	V_{\hat{\theta}}(\n) = \sum_{h \in \H}\, \tfrac{A_h^2}{n_h} - A_0,
\end{equation}
where $\n = (n_h,\, h \in \H)$ and $A_0,\, A_h > 0,\, h \in \H$ do not depend on $\n$. Among the most basic and common 
examples that give rise to the variance of the form \eqref{eq:var} is the {\em stratified $\pi$ estimator} of the population total 
with {\em simple random sampling without replacement} design in strata. This case yields in \eqref{eq:var}: $A_h = N_h S_h,\, 
h \in \H$, and $A_0 = \sum_{h \in \H}\, N_h S_h^2$, where $S_h$ denotes stratum standard deviation of study variable $y$ 
\citep[see, e.g.][equation 3.7.8, p.~103]{Sarndal}.

The values of the strata sample sizes $n_h,\, h \in \H$, are chosen by the sampler. They may be selected to minimize the 
variance \eqref{eq:var} at the admissible level of the total cost of the survey or to minimize the total cost of the survey subject 
to a fixed precision \eqref{eq:var}. The simplest total cost function is of the form:
\begin{equation}
	\label{eq:cost}
	c(\n) = c_0 + \sum_{h \in \H} c_h n_h,
\end{equation}
where $c_0$ is a fixed overhead cost and $c_h > 0$ is the cost of surveying one element in stratum $U_h,\, h \in \H$. For 
further references, see, e.g. \citet[Section 3.7.3, p.~104]{Sarndal} or \citet[Section 5.5, p.~96]{Cochran}. In this paper, we are 
interested in the latter strategy, i.e. the determination of the sample allocation $\n$ that minimizes total cost \eqref{eq:cost} 
under assumed fixed level of the variance \eqref{eq:var}. We also impose one-sided upper bounds on sample sizes in strata. 
Such optimization problem can be conveniently written in the language of mathematical optimization as Problem 
\ref{prob:min_cost}, in the definition of which we intentionally omit fixed overhead cost $c_0$ as it has no impact on the 
optimal solution to this problem.

\begin{problem}
	\label{prob:min_cost}
	Given a finite set $\H \neq \emptyset$ and numbers $A_0,\, A_h > 0,\, c_h > 0,\, M_h > 0$, such that $M_h \leq N_h,\, h \in 
	\H$, and $V \geq \sum_{h \in \H} \tfrac{A_h^2}{M_h} - A_0 \geq 0$,
	\begin{align}
		\underset{\x\, =\, (x_h,\, h \in \H)\, \in\, \R_+^{\card{\H}}}{\texteq{minimize ~\,}}  & \quad \sum_{h \in \H} c_h x_h 
		\label{eq:prob:min_cost:obj} \\
		\texteq{subject ~ to \quad\,\,\,}	& \quad \sum_{h \in \H} \tfrac{A_h^2}{x_h} - A_0 = V \label{eq:prob:min_cost:var} \\
		& \quad x_h \le M_h, \ind{h \in \H.} \label{eq:prob:min_cost:M}
	\end{align}
\end{problem}

To emphasize the fact that the optimal solution to Problem \ref{prob:min_cost} may not be an integer one, we denote the 
optimization variable by $\x$, not by $\n$. Non-integer solution can be rounded up in practice with the resulting variance 
\eqref{eq:var} being possibly near $V$, instead of the exact $V$. The upper bounds $M_h$ imposed on $x_h,\, h \in \H$, are 
natural since for instance the allocation with $x_h > N_h$ for some $h \in \H$ is impossible. We assume that $V \geq 
\sum_{h \in \H} \tfrac{A_h^2}{M_h} - A_0$, since otherwise, if $V < \sum_{h \in \H} \tfrac{A_h^2}{M_h} - A_0$, the problem is 
infeasible. We also note that in the case when $V = \sum_{h \in \H}\, \tfrac{A_h^2}{M_h} - A_0$, the solution is trivial, i.e.: 
$\x^* = (M_h,\, h \in \H)$. 

It is worth noting that in the definition of Problem \ref{prob:min_cost}, we require (through \eqref{eq:prob:min_cost:var}) that 
the variance defined in \eqref{eq:var} is equal to a certain fixed value, denoted as $V$, and not less than that, whilst, it might 
seem more favourable at first, to require the variance \eqref{eq:var} to be less than or equal to  $V$, especially given the 
practical context in which Problem \ref{prob:min_cost} arises (see Section \ref{sec:motivation} below). It is easy to see, 
however, that the objective function \eqref{eq:prob:min_cost:obj} and the variance constraint \eqref{eq:prob:min_cost:var} 
are of such a form that the minimum of \eqref{eq:prob:min_cost:obj} is achieved for a value that yields the allowable 
maximum of function \eqref{eq:var}, which is $V$. Thus, regardless of whether the variance constraint is an equality 
constraint or an inequality constraint, the optimal solution will be the same in both of these cases.

\bigskip
Our approach to the optimum allocation Problem \ref{prob:min_cost} will be twofold. First, in Section \ref{sec:optcon}, we 
make use of the Karush–Kuhn–Tucker conditions (see Appendix \ref{app:kkt}) to establish necessary and sufficient 
conditions, the so-called optimality conditions, for a solution to slightly reformulated optimization Problem 
\ref{prob:min_cost}, defined as a separate Problem \ref{prob:lower}. This task is one of the main objectives of this paper. 
Optimality conditions, which are often given as closed-form expressions, are fundamental to the analysis and development of 
effective algorithms for an optimization problem. Namely, algorithms recognize solutions by checking whether they satisfy 
various optimality conditions and terminate when such conditions hold. This elegant strategy has evident advantages over 
some alternative ad-hoc approaches, commonly used in survey sampling, which are usually tailored for a specific allocation 
algorithm being proposed. Next, in Section \ref{sec:lrna}, we precisely define the {\em LRNA} algorithm which solves Problem 
\ref{prob:lower} (and in consequence Problem \ref{prob:min_cost}) and based on the established optimality conditions we 
provide the formal proof of its optimality, which is the second main objective of this paper. 
To complement our work on this subject, we provide user-end function in \proglang{R} \citep[see][]{R} that implements the 
{\em LRNA}. This function is included in our package {\tt stratallo} \citep{stratallo}, which is published on the Comprehensive 
R Archive Network (CRAN) package repository.

\section{Motivation}
\label{sec:motivation}

Optimum sample allocation Problem \ref{prob:min_cost} is not only a theoretical problem, but it is also an issue of substantial 
practical importance. Usually, an increase in the number of samples entails greater costs of the data collection process. Thus, 
it is often demanded that total cost \eqref{eq:cost} be somehow minimized. On the other hand, the minimization of the cost 
should not cause significant reduction of the quality of the estimation, which can be measured by the variance \eqref{eq:var}. 
Hence, Problem  \ref{prob:min_cost} arises very naturally in the planning of sample surveys, when it is necessary to obtain an 
estimator $\hat{\theta}$ with some predetermined precision $V$ that ensures the required level of estimation quality, while 
keeping the overall cost as small as possible. Problem \ref{prob:min_cost} appears also in the context of optimum 
stratification and sample allocation between subpopulations in \citet{SkibickiWywiał} or \citet{Lednicki2}. The authors of the 
latter paper incorporate variance equality constraint into the objective function and then use numerical algorithms (for 
minimization of a non-linear multivariate function) to find the minimum of the objective function. If the solution found violates 
any of the inequality constraints, then the objective function is properly adjusted and the algorithm is re-run again. See also a 
related paper by \citet{WNB}, where the allocation under the constraint of the equal precision for estimation of the strata 
means was considered.

\bigskip
The problem of minimization of the total cost under constraint on stratified estimator's variance is well known in the domain 
literature. It was probably first formulated by Tore Dalenius in \citet{Dalenius1949, Dalenius1953} and later in his Ph.D. thesis 
\citet[Chapter 1.9, p.~19; Chapters 9.4 - 9.5, p.~199]{Dalenius}. Dalenius formed this allocation problem in the context of 
multicharacter (i.e. in the presence of several variables under study) stratified sampling (without replacement) and without 
taking into account upper-bounds constraints \eqref{eq:prob:min_cost:M}. He solved his problem with the use of simple 
geometric methods for the case of two strata and two estimated population means, indicating that the technique that was 
used is applicable also for the case with any number of strata and any number of variables. Among other resources that are 
worth mentioning are \citet{Yates} and \citet{Chatterjee}.

\cite{KokanKhan} considered a multicharacter generalization of Problem \ref{prob:min_cost} and proposed a procedure that 
leads to the solution of this problem. The proof of the optimality of the obtained solution given by the authors is not strictly 
formal and, similarly to Dalenius' work, is based solely on geometrical methods. The {\em LRNA} algorithm presented in 
Section \ref{sec:lrna} below, can be viewed as a special case of that Kokan-Khan's procedure for a single study variable. It is 
this method that is generally accepted as the one that solves Problem \ref{prob:min_cost} and is described in popular survey 
sampling textbooks such as, e.g. \citet[Remark 12.7.1, p.~466]{Sarndal} or \citet[Section 5.8, p.~104]{Cochran}. For earlier 
references, see \citet{Hartley} or \citet{Kokan1963}, who discussed how to use non-linear programming technique to 
determine allocations in multicharacter generalization of Problem \ref{prob:min_cost}. More recent references can be made to 
\citet{Bethel}, who proposed a closed form expression (in terms of Lagrange multipliers) for a solution to relaxed Problem 
\ref{prob:min_cost} without \eqref{eq:prob:min_cost:M}, as well as to \citet{HughesRao}, who obtained the solution to 
Problem \ref{prob:min_cost} by employing an extension of a result due to \citet{Thompson}. Eventually, for integer solution to 
Problem \ref{prob:min_cost}, we refer to \citet{KhanAhsan}.

\bigskip
We would like to note that the form of the transformation \eqref{eq:optvar_change} that we have chosen to convert Problem 
\ref{prob:min_cost} into a convex optimization Problem \ref{prob:lower} was not the only possible choice. An alternative 
transformation is for instance $z_h = \tfrac{1}{x_h},\, h \in \H$, which was used by \citet{KokanKhan} or \citet{HughesRao} in 
their approaches to (somewhat generalized) Problem \ref{prob:min_cost}. Nevertheless, as it turns out, transformation 
\eqref{eq:optvar_change} causes that induced Problem \ref{prob:lower} gains some interesting interpretation from the point 
of view of practical application. That is, if one treats $z_h$ as stratum sample size $x_h,\, h \in \H$, then Problem 
\ref{prob:lower} with $\Vt = n$ and $c_h = 1, h \in \H$, becomes a classical optimum allocation Problem \ref{prob:optalloc} 
with added one-sided lower-bounds constraints $z_h \geq m_h > 0,\, h \in \H$. Such allocation problem can be viewed as 
twinned to Problem \ref{prob:upper} and is itself interesting for practitioners. The lower bounds are necessary, e.g. for 
estimation of population strata variances $S_h^2,\, h \in \H$, which in practice are rarely known a priori. If they are to be 
estimated from the sample, it is required that at least $n_h \geq 2,\, h \in \H$. They also appear when one treats strata as 
domains and assigns upper bounds for variances of estimators of totals in domains. Such approach was considered, e.g. in 
\citet{Choudhry}, where the additional constraints $(\tfrac{1}{n_h} - \tfrac{1}{N_h}) N_h^2 S_h^2 \le R_h,\, h \in \H$, where 
$R_h,\, h \in \H$ are given constants, have been imposed. Obviously, this system of inequalities can be rewritten as 
lower-bounds constraints on $n_h$, i.e. $n_h \geq m_h = \tfrac{N_h^2 S_h^2}{R_h + N_h S_h^2},\, h \in \H$. The solution 
given in \citet{Choudhry} was obtained by the procedure based on the Newton-Raphson algorithm, a general-purpose 
root-finding numerical method. See also a related paper by \citet{WNB}, where the allocation under the constraint of the 
equal precision for estimation of the strata means was considered. The affinity between allocation Problem \ref{prob:lower} 
and Problem \ref{prob:upper} translates to significant similarities between the {\em LRNA} that solves Problem 
\ref{prob:lower} and the popular recursive Neyman allocation procedure, {\em RNA}, that solves Problem \ref{prob:upper} 
(see Appendix \ref{app:rna}). To emphasize these similarities, the name {\em LRNA} was chosen for the former.

In summary, the {\em LRNA}, formulated as in this work, solves two different but related problems of optimum sample 
allocation that are of a significant practical importance, i.e. Problem \ref{prob:min_cost} and Problem \ref{prob:lower}.

\section{Optimality conditions}
\label{sec:optcon}

In this section, we establish a general form of the solution to (somewhat reformulated) Problem \ref{prob:min_cost}, the 
so-called optimality conditions. For this problem, the optimality conditions can be derived reliably from the 
Karush–Kuhn–Tucker (KKT) conditions, first derivative tests for a solution in nonlinear programming to be optimal (see 
Appendix \ref{app:kkt} and the references given therein for more details). It is well known that for convex optimization 
problem (with some minor regularity conditions) the KKT conditions are not only necessary but also sufficient. Problem 
\ref{prob:min_cost} is however not a convex optimization problem because the equality constraint function $\sum_{h \in \H} 
\tfrac{A_h^2}{x_h} - A_0 - V$ of $\x = (x_h,\,h \in \H)$ is not affine and hence, the feasible set might not be convex. 
Nevertheless, it turns out that Problem \ref{prob:min_cost} can be easily reformulated to a convex optimization Problem 
\ref{prob:lower}, by a simple change of its optimization variable from $\x$ to $\z = (z_h,\, \in \H)$ with elements of the form:
\begin{equation}
	\label{eq:optvar_change}
	z_h := \tfrac{A_h^2}{c_h x_h}, \ind{h \in \H.}
\end{equation}

\begin{problem}
	\label{prob:lower}
	Given a finite set $\H \neq \emptyset$ and numbers $A_h > 0,\, c_h > 0,\, m_h > 0,\, h \in \H$, $\Vt \geq \sum_{h \in \H} c_h  
	m_h$,
	\begin{align}
		\underset{\z\, =\, (z_h,\, h \in \H)\, \in\, \R_+^{\card{\H}}}{\texteq{minimize ~\,}}  & \quad \sum_{h \in \H} \tfrac{A_h^2}{z_h} 
		\label{eq:prob:obj} \\
		\texteq{subject ~ to \quad\,\,\,}	& \quad \sum_{h \in \H} c_h z_h = \Vt \label{eq:prob:var} \\
		& \quad z_h \geq m_h, \ind{h \in \H.} \label{eq:prob:lower:m}
	\end{align}
\end{problem}

For Problem \ref{prob:lower} to be equivalent to Problem \ref{prob:min_cost} under transformation \eqref{eq:optvar_change}, 
parameters $m_h,\, h \in \H$, and $\Vt$ must be such that
\begin{equation}
	\label{eq:optvar_change_params}
	\begin{aligned}
		m_h &:= \tfrac{A_h^2}{c_h M_h}, \ind{h \in \H,} \\
		\Vt &:= V + A_0 \geq A_0,
	\end{aligned}
\end{equation}
where numbers $V,\, A_0,\, M_h,\, h \in \H$ are as in Problem \ref{prob:min_cost}. Nonetheless, as we explained at the end of 
Section \ref{sec:motivation} of this paper, Problem \ref{prob:lower} can be considered as a separate allocation problem, 
unrelated to Problem \ref{prob:min_cost}; that is Problem \ref{prob:optalloc} with added one-sided lower-bounds constraints. 
For this reason, the only requirements imposed on these parameters are those given in the definition of Problem 
\ref{prob:lower}.

\bigskip
The auxiliary optimization Problem \ref{prob:lower} is indeed a convex optimization problem as it is justified by Remark 
\ref{rem:prob_convex}.

\begin{remark}
	\label{rem:prob_convex}
	Problem \ref{prob:lower} is a convex optimization problem as its objective function $f: \R_+^{\card{\H}} \to \R_+$, 
	\begin{equation}
		f(\z)  = \sum_{h \in \H} \tfrac{A_h^2}{z_h}, \label{eq:rem_prob_convex:f}
	\end{equation}
	and inequality constraint functions $g_h: \R_+^{\card{\H}} \to \R$, 
	\begin{align}
		&g_h(\z) = m_h - z_h,  \ind{h \in \H,} \label{eq:rem_prob_convex:g}
	\end{align}	
	are convex functions, while the equality constraint function $w: \R_+^{\card{\H}} \to \R$, 
	\begin{equation*}
		w(\z) = \sum_{h \in \H} c_h z_h - \Vt
	\end{equation*}	
	is affine. More specifically, Problem \ref{prob:lower} is a convex optimization problem of a particular type in which 
	inequality constraint functions \eqref{eq:rem_prob_convex:g} are affine. See Appendix \ref{app:kkt} for the definition of the
	convex optimization problem.
\end{remark}

As we shall see in Theorem \ref{th:optcond}, the optimization Problem \ref{prob:lower} has a unique optimal solution. 
Consequently, due to transformation \eqref{eq:optvar_change} and given \eqref{eq:optvar_change_params}, vector
\begin{equation}
	\label{eq:redef_minsample_lower_sol}
	\x^* = \left(\tfrac{A_h^2}{c_h z^*_h},\, h \in \H \right)
\end{equation}	
is a unique optimal solution of Problem \ref{prob:min_cost}, where $\z^* = (z^*_h,\, h \in \H)$ is a solution to Problem 
\ref{prob:lower}. For this reason, for the remaining part of this work, our focus will be on a solution to Problem 
\ref{prob:lower}. We also note here that the solution to Problem \ref{prob:lower} is trivial in the case of $ \Vt = \sum_{h \in \H} 
c_h m_h$, i.e.: $\z^* = (m_h,\, h \in \H)$.

\bigskip
Before we establish necessary and sufficient optimality conditions for a solution to convex optimization Problem 
\ref{prob:lower}, we first define a set function $s$, which considerably simplifies notation and many calculations that are 
carried out in this and subsequent section.

\begin{definition}
	\label{def:s}
	Let $\H,\, A_h,\, c_h,\, m_h,\, h \in \H$, and $\Vt$ be as in Problem \ref{prob:lower}. Set function $s$ is defined as:
	\begin{equation}
		\label{eq:s}
		s(\L) = \frac{\Vt - \sum_{h \in \L} c_h m_h}{\sum_{h \in \H \setminus \L} A_h \sqrt{c_h}}, \ind{\L \subsetneq \H.}
	\end{equation}
\end{definition}

Below, we will introduce the notation of vector $\z^\L = (z_h^\L,\, h \in \H)$. It turns out that the solution to Problem 
\ref{prob:lower} is necessarily of the form \eqref{eq:Valloc} with the set $\L \subseteq \H$ defined implicitly through the 
inequality of a certain form given in Theorem \ref{th:optcond}.

\begin{definition}
	\label{def:Valloc}
	Let $\H,\, A_h,\, c_h,\, m_h,\, h \in \H,\, \Vt$ be as in Problem \ref{prob:lower} and let $\L \subseteq \H$. Vector  $\z^\L = 
	(z_h^\L,\, h \in \H)$ is defined as follows
	\begin{equation}
		\label{eq:Valloc}
		z_h^\L = 
		\begin{cases}
			m_h,										&\ind{h \in \L} \\
			\tfrac{A_h}{\sqrt{c_h}}\, s(\L)	&\ind{h \in \H \setminus \L.}
		\end{cases}
	\end{equation}	
\end{definition}

The following Theorem \ref{th:optcond} characterizes the form of the optimal solution to Problem \ref{prob:lower} and 
therefore is the key theorem of this paper.

\begin{theorem}[Optimality conditions]
	\label{th:optcond}
	The optimization Problem \ref{prob:lower} has a unique optimal solution. Point $\z^* = (z_h^*,\, h \in \H) \in 
	\R_+^{\card{\H}}$ is a solution to optimization Problem \ref{prob:lower} if and only if $\z^*= \z^{\L^*}$ with $\L^* \subseteq 
	\H$, such that one of the following two cases holds:
	\begin{enumerate}[wide, labelindent=0pt, leftmargin=*]
		\item[CASE I:] $\L^* \subsetneq \H$ and
		\begin{equation}
			\L^* = \left\{h \in \H:\, s(\L^*) \leq \tfrac{\sqrt{c_h} m_h}{A_h} \right\}, \label{eq:optcond_1} \\
		\end{equation}
		where set function $s$ is defined in \eqref{eq:s}.
		\item[CASE II:] $\L^* = \H$ and
		\begin{equation}
			\Vt = \sum_{h \in \H} c_h m_h. \label{eq:optcond_2}
		\end{equation}
	\end{enumerate}
\end{theorem}

\begin{proof}
	We first prove that the solution to Problem \ref{prob:lower} exists and it is unique. In the optimization Problem 
	\ref{prob:lower}, a feasible set $F := \{\z \in \R_+^{\card{\H}}: \eqref{eq:prob:var} \text{ and } \eqref{eq:prob:lower:m} 
	\text{ are satisfied} \}$ is non-empty as guaranteed by the requirement $\Vt \geq \sum_{h \in \H} c_h m_h$. The objective 
	function in \eqref{eq:prob:obj} attains its minimum on $F$ since it is a continuous function on $F$ and $F$ is closed and 
	bounded. Finally, the uniqueness of the solution is due to strict convexity of the objective function on the set $F$.
	
	\bigskip
	As mentioned at the beginning of Section \ref{sec:optcon}, the form of the solution to Problem \ref{prob:lower}, can be 
	derived from the KKT conditions (see Appendix \ref{app:kkt}). Following the notation of Remark \ref{rem:prob_convex}, 
	gradients of the objective function $f$ and constraint functions $w,\, g_h,\, h \in \H$, are as follows:
	\begin{equation*}
		\nabla f(\z) = \left(- \tfrac{A_h^2}{z_h^2},\, h \in \H \right), \quad
		\nabla w(\z) = (c_h,\, h \in \H), \quad 
		\nabla g_h(\z)  = -\underline{1}_h, \quad \z \in \R_+^{\card{\H}},
	\end{equation*}\\\\
	where $\underline{1}_h$ is a vector with all entries $0$ except the entry at index $h$, which is $1$. Consequently, 
	the KKT conditions \eqref{eq:kkt} assume the following form for the optimization Problem \ref{prob:lower}:
	\begin{align}
		-\tfrac{A_h^2}{z^{*2}_h} + \lambda c_h - \mu_h &= 0, \ind{h \in \H,} \label{eq:kkt_prob_stat} \\
		\sum_{h \in \H} c_h z_h^* - \Vt &= 0, \label{eq:kkt_prob_V} \\
		m_h - z^*_h &\leq 0, \ind{h \in \H,}  \label{eq:kkt_prob_ineq} \\
		\mu_h(m_h - z^*_h) &= 0,  \ind{h \in \H.}  \label{eq:kkt_prob_compl} 
	\end{align}
	Following Theorem \ref{th:kkt} and Remark \ref{rem:prob_convex}, in order to prove Theorem \ref{th:optcond}, it suffices 
	to show that there exist $\lambda \in \R$ and $\mu_h \ge 0,\, h \in \H$, such that \eqref{eq:kkt_prob_stat} - 
	\eqref{eq:kkt_prob_compl} are met for $\z^* = \z^{\L^*}$ with $\L^* \subseteq \H$ satisfying conditions of CASE I or CASE 
	II.
	
	\begin{enumerate}[wide, labelindent=0pt, leftmargin=*]
		\item[CASE I:] Following \eqref{eq:Valloc} and \eqref{eq:s}, we get
		\begin{equation*}
			\sum_{h \in \H}\, {c_h z_h^*} = \sum_{h \in \L^*} c_h m_h + \sum_{h \in \H \setminus \L^*}\, c_h 
			\tfrac{A_h}{\sqrt{c_h}}\, 
			s(\L^*) = \Vt,
		\end{equation*}
		and hence, the condition \eqref{eq:kkt_prob_V} is always satisfied.
		Let $\lambda = \tfrac{1}{s^2(\L^*)}$, where $s(\L^*) > 0$ is defined in \eqref{eq:s}, and
		\begin{equation}
			\mu_h = 
			\begin{cases}
				\lambda c_h - \tfrac{A_h^2}{m^2_h},	&\ind{h \in \L^*} \\
				0,	&\ind{h \in \H \setminus \L^*.}
			\end{cases}
		\end{equation}	
		Note that $\mu_h \geq 0,\, h \in \L^*$, due to \eqref{eq:optcond_1}. Then, the condition \eqref{eq:kkt_prob_stat} is 
		clearly satisfied. Inequalities \eqref{eq:kkt_prob_ineq} and equalities \eqref{eq:kkt_prob_compl} are trivial for $h \in 
		\L^*$ since ${z_h^*} = m_h$. For $h \in \H \setminus \L^*$, inequalities \eqref{eq:kkt_prob_ineq} follow from 
		\eqref{eq:optcond_1}, i.e. $\tfrac{A_h}{\sqrt{c_h}}\, s(\L^*) > m_h$, whilst \eqref{eq:kkt_prob_compl} hold true due to 
		$\mu_h = 0$.
		
		\item[CASE II:] Take arbitrary $\lambda \geq \max_{h \in \H}\, \tfrac{A_h^2}{m_h^2 c_h}$ and $\mu_h = \lambda c_h - 
		\tfrac{A_h^2}{m^2_h},\, h \in \H$. Note that $\mu_h \geq 0,\, h \in \H$. Then, \eqref{eq:kkt_prob_stat} - 
		\eqref{eq:kkt_prob_compl} are clearly satisfied for $(z^*_h,\, h \in \H) = (m_h,\, h \in \H)$, whilst \eqref{eq:kkt_prob_V} 
		follows after referring to \eqref{eq:optcond_2}.
	\end{enumerate}
\end{proof}

Theorem \ref{th:optcond} gives the general form of the optimum solution up to specification of the set $\L^* \subseteq \H$ 
that corresponds to the optimal solution $\z^* = \z^{\L^*}$. The issue of how to identify this set is the subject of the next 
section of this paper.

\section{Recursive Neyman algorithm under lower-bounds constraints}
\label{sec:lrna}

In this section, we formalize the definition of the existing algorithm, termed here {\em LRNA}, solving Problem \ref{prob:lower} 
and provide a formal proof of its optimality. The proof given is based on the optimality conditions formulated in Theorem 
\ref{th:optcond}.

\begin{algorithm}[H]
	\caption{{\em LRNA}}
	\textbf{Input:} $\H,\, (A_h)_{h \in \H},\, (c_h)_{h \in \H},\, (m_h)_{h \in \H},\, \Vt$.
	\begin{algorithmic}[1]
		\Require $A_h > 0,\, c_h > 0,\, m_h > 0,\, h \in \H$, $\Vt \geq \sum_{h \in \H}\, c_h m_h$.
		\State Let $\L = \emptyset$.
		\State\label{alg:lrna:step_R}Determine $\Lt = \left\{h \in \H \setminus \L:\, \tfrac{A_h}{\sqrt{c_h}} s(\L) \leq m_h \right\}$,
		where function $s$ is defined in \eqref{eq:s}.
		\State\label{alg:lrna:step_check} If {$\Lt = \emptyset$}, go to \footnotesize Step \ref{alg:lrna:step_return} \normalsize. 
		Otherwise, update $\L \gets \L \cup \Lt$  and go to \footnotesize Step \ref{alg:lrna:step_R} \normalsize.
		\State\label{alg:lrna:step_return}Return $\z^* = (z^*_h,\, h \in \H)$ with
		$z^*_h =
		\begin{cases}
			m_h,											& h \in \L \\
			\tfrac{A_h}{\sqrt{c_h}}\, s(\L),	& h \in \H \setminus \L.
		\end{cases}
		$
	\end{algorithmic}
\end{algorithm}

\begin{theorem}
	\label{th:lrna}
	The {\em LRNA} provides an optimal solution to optimization Problem \ref{prob:lower}.
\end{theorem}


Before we prove Theorem \ref{th:lrna}, we first reveal certain monotonicity property of set function $s$, defined in 
\eqref{eq:s}, that will be essential to the proof of this theorem.

\begin{lemma}
	\label{lemma:s_mono} 
	Let $\A \subseteq \B \subsetneq \H$. Then
	\begin{equation}
		\label{eq:s_mono}
		s(\A) \geq s(\B)  \quad \Leftrightarrow \quad s(\A) \sum_{h \in \B \setminus \A}\, A_h \sqrt{c_h} \leq \sum_{h \in \B 
		\setminus \A}\, c_h m_h,
	\end{equation}
	where set function $s$ is defined in \eqref{eq:s}.
\end{lemma}

\begin{proof}
	Clearly, for any $\alpha \in \R,\, \beta \in \R, \delta \in \R$, $\gamma \in \R_+$, such that $\gamma + \delta > 0$, we have 
	\begin{equation}
		\label{eq:smono_gene}
		\tfrac{\alpha + \beta}{\gamma + \delta} \geq \tfrac{\alpha}{\gamma} \quad \Leftrightarrow \quad \tfrac{\alpha + 
		\beta}{\gamma + \delta} \delta \leq \beta.
	\end{equation}
	To prove \eqref{eq:s_mono}, take
	\begin{align*}
		\alpha	  &=  \Vt - \sum_{h \in \B} c_h m_h & \beta & = \sum_{h \in \B \setminus \A} c_h m_h \\
		\gamma &= \sum_{h \in \H \setminus \B} A_h \sqrt{c_h} & \delta &= \sum_{h \in \B \setminus \A} A_h \sqrt{c_h}.
	\end{align*}
	Then, $\tfrac{\alpha}{\gamma} = s(\B),\, \tfrac{\alpha + \beta}{\gamma + \delta} = s(\A)$, and hence \eqref{eq:s_mono} 
	holds as an immediate consequence of \eqref{eq:smono_gene}.
\end{proof}

We are now ready to give the proof of Theorem \ref{th:lrna}.

\begin{proof}[Proof of Theorem \ref{th:lrna}]
	Let $\L_r,\, \Lt_r$ denote sets $\L$ and $\Lt$ respectively, as in the $r$-th iteration of the {\em LRNA} algorithm, at the 
	moment after {\footnotesize Step \ref{alg:lrna:step_R}} and before {\footnotesize Step \ref{alg:lrna:step_check}}. The 
	iteration index $r$ takes on values from set $\{1, \ldots, r^*\}$, where $r^* \geq 1$ indicates the final iteration of the 
	algorithm. Under this notation, we have $\L_1 = \emptyset$ and in general for subsequent iterations, if any (i.e. if $r^* \geq 
	2$), we get
	\begin{equation}
		\label{eq:lrna_Vsum}
		\L_r = \L_{r-1} \cup \Lt_{r-1} = \bigcup\limits_{i = 1}^{r-1} \Lt_i, \ind{r = 2, \ldots, r^*.}
	\end{equation}
	
	To prove Theorem \ref{th:lrna}, we have to show that: 
	\begin{enumerate}[label=(\Roman*), nosep, labelsep=5pt]
		\item the algorithm terminates in a finite number of iterations, i.e. $r^* < \infty$, \label{rna_proof_I}
		\item the solution computed at $r^*$ is optimal. \label{rna_proof_II}
	\end{enumerate}
	
	The proof of \ref{rna_proof_I} is relatively straightforward. In every iteration $r = 2, \ldots, r^* \geq 2$, the domain of 
	discourse for $\Lt_r$ at {\footnotesize Step \ref{alg:lrna:step_R}} is $\H \setminus \L_r = \H \setminus \bigcup\limits_{i = 
	1}^{r-1} \Lt_i$, where $\Lt_i \neq \emptyset,\, i = 1, \ldots, r - 1$. Therefore, in view of {\footnotesize Step 
	\ref{alg:lrna:step_check}}, we have that $r^* \leq \card{\H} + 1 < \infty$, where $r^* = \card{\H} + 1$ if and only if 
	$\card{\Lt_r} = 1$ for each $r = 1, \ldots, r^* - 1$. In words, the algorithm terminates in at most $\card{\H} + 1$ iterations.
	
	In order to prove \ref{rna_proof_II}, following Theorem \ref{th:optcond}, it suffices to show that for $\L_{r^*} \subsetneq 
	\H$ (CASE I), for all $h \in \H$,
	\begin{equation}
		\label{eq:lrna_proof_optcond_1}
		h \in \L_{r^*} \quad \Leftrightarrow \quad \tfrac{A_h}{\sqrt{c_h}}\, s(\L_{r^*}) \leq m_h,
	\end{equation}
	and for $\L_{r^*} = \H$ (CASE II):
	\begin{equation}
		\label{eq:lrna_proof_optcond_2}
		\Vt = \sum_{h \in \H} c_h m_h.
	\end{equation}\\

	We first note that the construction of the algorithm ensures that $\L_r \subsetneq \H$ for $r = 1, \ldots, r^* - 1$, $r^* \geq 
	2$, and therefore $s(\L_r)$ for such $r$ is well-defined.\\
	\begin{enumerate}[wide, labelindent=0pt, leftmargin=*]
		\item[CASE I.] The $s(\L_{r^*})$ is well-defined since in this case $\L_{r^*} \subsetneq \H$.
		\begin{enumerate}[wide, labelindent=0pt, leftmargin=*, font=\itshape]
			\item[Necessity:] For $r^* = 1$, we have $\L_{r^*} = \emptyset$ and hence, the right-hand side of equivalence 
			\eqref{eq:lrna_proof_optcond_1} is trivially met. Let $r^* \geq 2$. By \footnotesize Step \ref{alg:lrna:step_R} 
			\normalsize of the {\em LRNA}, we have
			\begin{equation}
				\label{eq:lrna_proof_less_m}
				\tfrac{A_h}{\sqrt{c_h}}\, s(\L_r) \leq m_h, \ind{h \in \Lt_r,}
			\end{equation}
			for every $r = 1, \ldots, r^* - 1$. Multiplying inequalities \eqref{eq:lrna_proof_less_m} sidewise by $c_h$ and 
			summing over $h \in \Lt_r$, we get the right-hand side of equivalence \eqref{eq:s_mono} with $\A = \L_r$ and $\B = 
			\L_r \cup \Lt_r  = \L_{r+1} \subsetneq \H$. Then, by Lemma \ref{lemma:s_mono}, the first inequality in 
			\eqref{eq:s_mono} follows. Consequently,
			\begin{equation}
				\label{eq:lrna:proof:s_monotone}
				s(\L_1) \geq \ldots \geq s(\L_{r^*}).
			\end{equation}
			Now, assume that $h \in \L_{r^*} = \bigcup\limits_{r = 1}^{r^*-1} \Lt_r$. Thus, $h \in \Lt_r$ for some $r \in \{1, \ldots, 
			r^*-1\}$, and again, using \footnotesize Step \ref{alg:lrna:step_R} \normalsize of the {\em LRNA}, we get 
			$\tfrac{A_h}{\sqrt{c_h}}\, s(\L_r) \leq m_h$. Consequently, \eqref{eq:lrna:proof:s_monotone} yields 
			$\tfrac{A_h}{\sqrt{c_h}}\, s(\L_{r^*}) \leq m_h$.
			
			\item[Sufficiency:] The proof is by establishing a contradiction. Assume that $\tfrac{A_h}{\sqrt{c_h}}\, s(\L_{r^*}) \leq 
			m_h$ and $h \notin \L_{r^*}$. On the other hand, \footnotesize Step \ref{alg:lrna:step_check} \normalsize of the {\em 
			LRNA} yields $\tfrac{A_h}{\sqrt{c_h}}\, s(\L_{r^*}) > m_h$ for $h \in \H \setminus \L_{r^*}$ (i.e. $h \not \in \L_{r^*}$), 
			which contradicts the assumption.
		\end{enumerate}
		\newpage 
		\item[CASE II.] Note that in this case it must be that $r^* \geq 2$. Following \footnotesize Step \ref{alg:lrna:step_R} 
		\normalsize of the {\em LRNA}, the only possibility is that for all $h \in \H \setminus \L_{r^* - 1}$,
		\begin{equation}
			\label{eq:lrna:proof:case_2_r1}
			\tfrac{A_h}{\sqrt{c_h}}\, s(\L_{r^* - 1}) = \tfrac{A_h}{\sqrt{c_h}}\, \frac{\Vt - \sum_{i \in \L_{r^* - 1}} c_i m_i}{\sum_{i \in 
			\H \setminus \L_{r^* - 1}} A_i \sqrt{c_i}} \leq m_h.
		\end{equation}
		Multiplying both sides of inequality \eqref{eq:lrna:proof:case_2_r1} by $c_h$, summing it sidewise over $h \in \H 
		\setminus \L_{r^* - 1}$, we get $\Vt \leq \sum_{i \in \H} c_i m_i$, which, when combined with the requirement $\Vt \geq 
		\sum_{h \in \H} c_h m_h$, yields $\Vt = \sum_{h \in \H} c_h m_h$, i.e. \eqref{eq:lrna_proof_optcond_2}.
	\end{enumerate}
	
\end{proof}

Following Theorem \ref{th:lrna} and given \eqref{eq:redef_minsample_lower_sol}, the optimal solution to Problem 
\ref{prob:min_cost} is vector $\x^* = (x^*_h,\, h \in \H)$ with elements of the form:
\begin{equation}
	x^*_h = 
	\begin{cases}
		M_h, & \ind{h \in \L^*} \\
		\tfrac{A_h}{\sqrt{c_h}}\, \tfrac{\sum_{i \in \H \setminus \L^*} A_i \sqrt{c_i}}{V + A_0 - \sum_{i \in \L^*} 
		\tfrac{A_i^2}{M_i}}, & \ind{h \in \H \setminus \L^*,}
	\end{cases}
\end{equation}	
where $\L^* \subseteq \H$ is determined by the {\em LRNA}.

\section{Final remarks and conclusions}
\label{sec:rem}

Within this work we formulated the optimality conditions for an important problem of minimum cost allocation under 
constraints on stratified estimator's variance and maximum samples sizes in strata. This allocation problem was defined in 
this paper as Problem \ref{prob:min_cost} and converted to Problem \ref{prob:lower} through transformation 
\eqref{eq:optvar_change} and under \eqref{eq:optvar_change_params}.  Based on the established optimality conditions, we 
provided a formal and compact proof of the optimality of the {\em LRNA} algorithm that solves the allocation problem  
mentioned. As already outlined at the end of Section \ref{sec:motivation} of this paper, Problem \ref{prob:lower} can be 
viewed in two ways, each of which is of a great practical importance. That is, apart from its primary interpretation as the 
problem of minimizing the total cost under given constraints, it can also be perceived as the problem of minimizing the 
stratified estimator's variance under constraint on total sample size (i.e. Problem \ref{prob:optalloc}) and constraints imposed 
on minimum sample sizes in the strata. For this reason, all the results of this work established in relation to Problem 
\ref{prob:lower} (i.e. optimality conditions and the {\em LRNA}) are of such a twofold nature. 

\bigskip
For the reasons mentioned at the end of Section \ref{sec:motivation}, the {\em LRNA} can be considered as a counterparty to 
the {\em RNA}. This resemblance is particularly desirable, given the popularity, simplicity as well as relatively high 
computational efficiency of the latter algorithm. Among the alternative approaches that could potentially be adapted to solve 
Problem \ref{prob:lower} are the ideas that underlay the existing algorithms dedicated to Problem \ref{prob:upper}, i.e.: {\em 
SGA} \citep{WWW, SG} and {\em COMA} \citep{WWW}. For integer-valued algorithms dedicated to Problem \ref{prob:lower} 
with added upper-bounds constraints, see \cite{Friedrich, Wright2017, Wright2020}. Nevertheless, it should be noted that 
integer-valued algorithms are typically relatively slow compared to not-necessarily integer-valued algorithms. As pointed out 
in \citet{Friedrich}, computational efficiency of integer-valued allocation algorithms becomes an issue for cases with "many 
strata or when the optimal allocation has to be applied repeatedly, such as in iterative solutions of stratification problems". 

\bigskip
Finally, we would like to emphasize that the optimality conditions established in Theorem \ref{th:optcond} can be used as a 
baseline for the development of new algorithms that provide solution to the optimum allocation problem considered in this 
paper. For instance, such algorithms could be derived by exploiting the ideas embodied in {\em SGA} or {\em COMA}, 
dedicated to Problem \ref{prob:upper} as indicated above.

\bigskip
Theoretical results obtained in this paper are complemented by the \proglang{R}-implementation \citep{R} of the {\em LRNA}, 
which we include in our publicly available package {\tt stratallo} \citep{stratallo}.

\section*{Acknowledgements}

I am very grateful to Jacek Wesołowski from Warsaw University of Technology, my research supervisor, for his patient 
guidance on this research work. Many thanks to Robert Wieczorkowski from Statistics Poland for advice and explanations on 
the topic of optimum stratification. I would also like to thank Reviewers for taking the necessary time and effort to review the 
manuscript. In particular, I express my gratitude to the second of the Reviewers for his expertise, valuable suggestions and 
for pointing to the existing papers, particularly important from the point of view of the subject I am addressing in this work.

\everypar = {
	\parindent=0pt
	\hangindent=8mm
	\hangafter=1
}\noindent 

\bibliography{lrna_arxiv_20231208}

\newpage

\section*{APPENDICES}
\appendix

\section{Recursive Neyman allocation}
\label{app:rna}

The classical problem of optimum sample allocation is described e.g. in \citet[Section 3.7.3, p.~104]{Sarndal}. It can be 
formulated in the language of mathematical optimization as Problem \ref{prob:optalloc}.
\begin{problem}
	\label{prob:optalloc}
	Given a finite set $\H \neq \emptyset$ and numbers $A_h > 0,\, h \in \H,\, 0 < n \leq N$,
	\begin{align*}
		\underset{\x\, = \, (x_h,\, h \in \H)\, \in\, \R_+^{\card{\H}}}{\texteq{minimize ~\,}}  & \quad \sum_{h \in \H} 
		\tfrac{A_h^2}{x_h} \\
		\texteq{subject ~ to \quad\,\,\,}	& \quad \sum_{h \in \H} x_h = n.
	\end{align*}
\end{problem}

The solution to Problem \ref{prob:optalloc} is $\x^* = (x^*_h,\, h \in \H)$ with elements of the form:
\begin{equation*} 
	x_h^* = A_h\, \frac{n}{\sum_{i \in \H} A_i}, \ind{h \in \H.}
\end{equation*}

It was established by \citet{Tschuprow1923a, Tschuprow1923b} and \citet{Neyman} for {\em stratified $\pi$ estimator} of the 
population total with {\em simple random sampling without replacement} design in strata, \linebreak in the case of which $A_h 
= N_h S_h,\, h \in \H$, where $S_h$ denotes stratum standard deviation of  a given study variable. See, e.g. \citet[Section 
3.7.4.i., p.~106]{Sarndal} for more details.

The recursive Neyman allocation algorithm, denoted here as {\em RNA}, is a well-established allocation procedure that finds a 
solution to the classical optimum sample allocation Problem \ref{prob:optalloc} with added one-sided upper-bounds 
constraints, defined here as Problem \ref{prob:upper}.
\begin{problem}
	\label{prob:upper}
	Given a finite set $\H \neq \emptyset$ and numbers $A_h > 0,\, M_h > 0$, such that $M_h \leq N_h,\, h \in \H$, and $0 < n 
	\leq \sum_{h \in \H}\, M_h$, 
	\begin{align*}
		\underset{\x\, =\, (x_h,\, h \in \H)\, \in\, \R_+^{\card{\H}}}{\texteq{minimize ~\,}}  & \quad \sum_{h \in \H} \tfrac{A_h^2}{x_h} 
		\\
		\texteq{subject ~ to \quad\,\,\,}	& \quad \sum_{h \in \H} x_h = n \\
		& \quad x_h \le M_h,  \ind{h \in \H.}
	\end{align*}
\end{problem}
\begin{algorithm}[H]
		\caption{{\em RNA}}
		\vspace{0.3cm}\textbf{Input:} $\H,\, (A_h)_{h \in \H},\, (M_h)_{h \in \H},\, n$.
	\begin{algorithmic}[1]
		\Require $A_h > 0,\, M_h > 0,\, h \in \H$, $0 < n \leq \sum_{h \in \H}\, M_h$.
		\State Let $\U = \emptyset$.
		\State\label{alg:rna:step_first}Determine $\Ut = \left\{h \in \H \setminus \U:\, A_h \,  \frac{n - \sum_{i \in \U} M_i}{\sum_{i 
		\in \H \setminus \U} A_i} \ge M_h \right\}$.
		\State If {$\Ut = \emptyset$}, go to \footnotesize Step \ref{alg:rna:step_return} \normalsize. Otherwise, update $\U \gets 
		\U \cup \Ut$, and go to \footnotesize Step \ref{alg:rna:step_first} \normalsize.
		\State\label{alg:rna:step_return}Return $\x^* = (x^*_h,\, h \in \H)$ with
		$x_h^* = 
		\begin{cases}
			M_h,		& \ind{h \in \U} \\
			A_h\, \frac{n - \sum_{i \in \U} M_i}{\sum_{i \in \H \setminus \U} A_i}, & \ind{h \in \H \setminus \U.}
		\end{cases}
		$
	\end{algorithmic}
\end{algorithm}
For more information on this recursive procedure see \citet[Remark 12.7.1, p.~466]{Sarndal} and \citet{WWW} for the proof 
of its optimality.\\\\

\section{Convex optimization scheme and the KKT conditions}
\label{app:kkt}

A convex optimization problem is an optimization problem in which the objective function is a convex function and the 
feasible set is a convex set. In standard form it is written as
\begin{equation}
	\label{prob:convex}
	\begin{split}
		\underset{\x\, \in\, \D}{\texteq{minimize ~\,}} & \quad f(\x) \\
		\texteq{subject ~ to}  & \quad w_i(\x) = 0, \quad {i = 1, \ldots, k} \\
		& \quad g_j(\x) \le 0, \quad{j = 1, \ldots, \ell,}\\\\
	\end{split}
\end{equation}

where $\D \subseteq \R^p,\, p \in \N_+$, the objective function $f:\D_f \subseteq \R^p \to \R$ and inequality constraint 
functions $g_j: \D_{g_j} \subseteq \R^p \to \R,\, j = 1, \ldots, \ell$, are convex, whilst equality constraint functions $w_i: 
\D_{w_i} \subseteq \R^p \to \R,\, i = 1, \ldots, k$, are affine. Here, $\D = \D_f \cap \bigcap_{i=1}^{k} \D_{w_i} \cap 
\bigcap_{j=1}^{\ell} \D_{g_j}$ denotes a common domain of all the functions. Point $\x \in \D$ is called {\em feasible} if it 
satisfies all of the constraints, otherwise the point is called {\em infeasible}. An optimization problem is called {\em feasible} if 
there exists $\x \in \D$ that is {\em feasible}, otherwise the problem is called {\em infeasible}.

\bigskip
In the context of the optimum allocation Problem \ref{prob:lower} discussed in this paper, we are interested in a particular 
type of the convex problem, i.e. \eqref{prob:convex} in which all inequality constraint functions $g_j,\, j = 1, \ldots, \ell$, are 
affine. It is well known, see, e.g. the monograph \citet{Boyd}, that the solution for such an optimization problem can be 
identified through the set of equations and inequalities known as the Karush-Kuhn-Tucker (KKT) conditions, which in this 
case are not only necessary but also sufficient.

\begin{theorem}[KKT conditions for convex optimization problem with affine inequality constraints]
	\label{th:kkt}
	A point $\x^* \in \D \subseteq \R^p$ is a solution to the convex optimization problem \eqref{prob:convex} in which 
	functions $g_j,\, j = 1, \ldots, \ell$, are affine if and only if there exist numbers $\lambda_i \in \R$, $i = 1, \ldots, k$, and 
	$\mu_j \geq 0$, $j = 1, \ldots, \ell$, called KKT multipliers, such that
	\begin{equation}
		\label{eq:kkt}
		\begin{gathered}
			\nabla f(\x^*)+\sum_{i=1}^k \lambda_i \nabla w_i(\x^*) + \sum_{j=1}^\ell \mu_j \nabla g_j(\x^*) = \zero \\
			w_i(\x^*) = 0, \ind{i = 1, \ldots, k} \\
			g_j(\x^*) \le 0, \ind{j = 1, \ldots, \ell} \\
			\mu_j g_j(\x^*) = 0, \ind{j = 1, \ldots, \ell.}
		\end{gathered}
	\end{equation}
\end{theorem}

\end{document}